\documentclass[12pt]{amsart}
\usepackage{amsthm}
\usepackage{amsmath,amssymb,amsfonts,marvosym,amscd}
\usepackage{tikz}
\usepackage{tikz-cd}\usepackage{wasysym}  
\usepackage{pifont} 
\usepackage{graphicx}
\usepackage{psfrag}
\usepackage{datetime}
\usepackage{enumerate}

\usepackage{geometry} 
\usepackage[all]{xy}
\usepackage{hyperref}  
  
\let\emptyset\varnothing  
\title[Higher supp. tilting I: higher Auslander alg. of lin. oriented type $A$]{Higher support tilting I: higher Auslander algebras of linearly oriented type $A$}
\author{Jordan McMahon}
\begin{document}

\newtheorem{lm}{Lemma}[section]
\newtheorem{prop}[lm]{Proposition}
\newtheorem{conj}[lm]{Conjecture}
\newtheorem{cor}[lm]{Corollary}
\newtheorem{theorem}[lm]{Theorem}

\theoremstyle{definition}
\newtheorem{defn}{Definition}
\newtheorem{eg}{Example}
\newtheorem{remark}{Remark}
\newtheorem{qu}{Question}

\begin{abstract}
For a path algebra $A$ over a quiver $Q$, there are bijections between the support-tilting modules of $A$, torsion classes in $\mathrm{mod}(A)$ and wide subcategories in $\mathrm{mod}(A)$; these are part of the Ingalls-Thomas bijections. As a blueprint for further study, we show how these bijections manifest themselves for higher Auslander algebras of linearly oriented type $A$. In particular, we introduce a higher analogue of torsion classes in $d$-representation-finite algebras.
\end{abstract}

\maketitle
\tableofcontents

\newcommand{\sub}{\underline}
\newcommand{\ra}{\rightarrow}
\newcommand{\mc}{\mathcal}
\newcommand{\blp}{^\prime}
\newcommand{\mcm}{\operatorname{\sub{MCM}(A)}}
\newcommand{\ba}{\begin{enumerate}[(a)]}

\section{Introduction}
Support-tilting modules were first studied by Ringel \cite{ringusup}. Their importance was highlighted by Ingalls-Thomas \cite{it}, who showed the following result.

\begin{theorem}\cite[Theorem 1.1]{it}\label{crowd}
For a finite-dimensional algebra $A=kQ$ that is hereditary and representation finite, there are bijections between the following objects:
\begin{itemize}
\item Isomorphism classes of basic support-tilting $A$-modules.
\item Torsion classes in $\mathrm{mod}(A)$.
\item Wide subcategories in $\mathrm{mod}(A)$.
\end{itemize}
\end{theorem}

This correspondence forms part of the Ingalls-Thomas bijections and was generalised to all representation-finite finite-dimensional algebras by Marks-{\v{S}}{\`t}ov{\'\i}{\v{c}}ek \cite{ms}. Included in the Ingalls-Thomas bijections were also clusters in the acyclic cluster algebra whose initial seed is in given by $Q$. A consequence was that support-tilting modules were able to capture the behaviour of clusters, and this led to further study. Significantly, support $\tau$-tilting theory \cite{air}, \cite{ir} was able capture this behavior more generally, and has seen much activity in recent years, see for example \cite{adachi}, \cite{asai}, \cite{bst}, \cite{dij},  \cite{dirrt}, \cite{irrt}, \cite{irtt}, \cite{jassored}, \cite{ppp}.

A natural question to ask is whether similar results are true in the context of higher Auslander-Reiten theory, as introduced by Iyama in \cite{iy1}, \cite{iy2}. The limiting factor is how the support of a module behaves: for an algebra $A$ with $d$-cluster-tilting subcategory $\mathcal{C}\subseteq \mathrm{mod}(A)$ we would like to be able to transfer information about $\mathcal{C}$ to this support. So to most easily replicate the theory of support tilting, we will only consider modules whose support is determined by an idempotent ideal $I$ such that $\mathcal{C}\cap \mathrm{mod}(A/I)\subseteq \mathrm{mod}(A/I)$ is a $d$-cluster-tilting subcategory. Any such module that is in addition $d$-tilting as an $A/I$-module will be said to be a \emph{proper $d$-support-tilting}. 

Another concept that needs to be generalised is that of a torsion class. Suppose that $\mathcal{T}\subseteq \mathcal{C}$ is an additive subcategory such that for any $d$-exact sequence $$0\rightarrow M_0\rightarrow M_1\rightarrow \cdots \rightarrow M_{d+1}\rightarrow M_{d+2}=0,$$ and for any $1\leq i\leq d+1$, if both $M_{i-1} \in\mathcal{T}$ and $M_{i+1}\in\mathcal{T}$ then so is $M_i$. In this case $\mathcal{T}$ is a \emph{$d$-strong torsion class}. These concepts are able to extend the classical behaviour of torsion classes to the class of $d$-representation-finite algebras whose $d$-cluster-tilting subcategories are almost directed (see Definition \ref{almost}). This class of algebras includes the higher Auslander algebras of linearly oriented type $A$.

\begin{theorem}[Theorem \ref{elso}]
Let $A$ be a finite-dimensional algebra with a $d$-cluster-tilting subcategory $\mathcal{C}\subseteq \mathrm{mod}(A)$ such that $\mathrm{gl.dim}(A)\leq d$ and $\mathcal{C}$ is almost directed. For a given proper support-$d$-tilting $A$-module $T$, let $\mathcal{T}:=\mathrm{Fac}(T)\cap \mathcal{C}$. Then the following hold:
\begin{enumerate}
\item $\mathcal{T}$ is a $d$-strong torsion class.
\item For all $M\in \mathcal{C}$, there is a module $F_M\in\mathrm{mod}(A)$ and an exact sequence
$$0\rightarrow T_1\rightarrow T_2\rightarrow \cdots\rightarrow T_d\rightarrow M \rightarrow F_M\rightarrow 0$$
such that $T_i\in\mathcal{T}$ for all $1\leq i\leq d$ and $$\mathcal{T}=\{T\in\mathcal{C}|\mathrm{Hom}_A(T,F_M)=0\ \forall M\in\mathcal{C}\}.$$ 
\end{enumerate}
\end{theorem}

A $d$-strong torsion class is \emph{standard} if it can be described as $\mathrm{Fac}(T)\cap\mathcal{C}$ for a proper support-$d$-tilting $A$-module $T$. Recently, wide subcategories were defined for $d$-abelian categories by Herschend-J\o rgensen-Vaso \cite{hjv}. One might hope that there is a bijection between wide subcategories and proper support-$d$-tilting modules. However in higher dimensions we must instead consider certain combinations of wide subcategories called \emph{resonant collections} (see Definition \ref{resonant}). This allows us to generalise Theorem \ref{crowd} to the following extent.

\begin{theorem}[Theorem \ref{masod}]
Let $A$ be a $d$-Auslander algebra of linearly oriented type $A_n$ with unique $d$-cluster-tilting subcategory $\mathcal{C}$. Then there are bijections between the following:

\begin{itemize}
\item Proper support-$d$-tilting $A$-modules. 
\item Standard $d$-strong torsion classes in $\mathcal{C}$. 
\item Resonant collections of wide subcategories in $\mathcal{C}$.
\item Standard $d$-strong torsion-free classes in $\mathcal{C}$.
\item Coresonant collections of wide subcategories in $\mathcal{C}$.
\end{itemize}
\end{theorem}

\section{Background and notation}

Consider a finite-dimensional algebra $A$ over a field $k$, and fix a positive integer $d$. An $A$-module will mean a finitely-generated left $A$-module; by $\mathrm{mod}(A)$ we denote the category of finite-dimensional left $A$-modules. The functor $D=\mathrm{Hom}_k(-,k)$ defines a duality, and we set $\tau_d=\tau\circ \Omega^{d-1}$ to be the $d$-Auslander-Reiten translation.  For an $A$-module $M$, let $\mathrm{add}(M)$ be the full subcategory of $\mathrm{mod}(A)$ composed of all $A$-modules isomorphic to direct summands of finite direct sums of copies of $M$. 

Define the \emph{dominant dimension} of $A$ $\mathrm{dom.dim}(A)$ to be the number $n$ such that for a minimal injective resolution of $A$:
$$0\rightarrow A \rightarrow I_0\rightarrow \ldots \rightarrow  I_{n-1}\rightarrow I_{n}\rightarrow \cdots$$ the modules $I_0,\cdots, I_{n-1}$ are projective-injective and $I_{n}$ is not projective.
A subcategory $\mathcal{C}$ of $\mathrm{mod}(A)$ is \emph{precovering} if for any $M\in \mathrm{mod}(A)$ there is an object $C_M\in\mathcal{C}$ and a morphism $f:C_M\rightarrow M$ such that for any morphism $X\rightarrow M$ with $X\in \mathcal{C}$ factors through $f$; that there is a commutative diagram:

$$\begin{tikzcd} 
\ &X\arrow{d}\arrow[dotted]{dl}[above]{\exists}\\
C_M\arrow{r}{f}&M
\end{tikzcd}$$

The dual notion of precovering is \emph{preenveloping}.  A subcategory $\mathcal{C}$ that is both precovering and preenveloping is called \emph{functorially finite}. For a finite-dimensional algebra $A$, a functorially-finite subcategory $\mathcal{C}$ of $\mathrm{mod}(A)$ is a \emph{$d$-cluster-tilting subcategory} if it satisfies the following conditions:
\begin{itemize}
\item $\mathcal{C}=\{M\in\mathrm{mod}(A)|\mathrm{Ext}^i_A(\mathcal{C},M)=0 \ \forall\ 0<i<d\}$.
\item  $\mathcal{C}=\{M\in\mathrm{mod}(A)|\mathrm{Ext}^i_A(M,\mathcal{C})=0 \ \forall\ 0<i<d\}$.
\end{itemize}
If there exists a $d$-cluster-tilting subcategory $\mathcal{C}\subseteq \mathrm{mod}(A)$, then $(A,\mathcal{C})$ is a \emph{$d$-homological pair} in the sense of \cite{hjv}. If $A$ is a finite-dimensional algebra such that $(A,\mathcal{C})$ is a $d$-homological pair and $\mathrm{gl.dim}(A)\leq d$, then $A$ is \emph{$d$-representation finite} in the sense of \cite{io}. The class of $d$-representation-finite algebras were characterised by Iyama as follows.

\begin{theorem}\cite[Proposition 1.3, Theorem 1.10]{iy2}\label{artheory}
Let $A$ be a finite-dimensional algebra such that $\mathrm{gl.dim}(A)\leq d$. Then there is a unique $d$-cluster-tilting subcategory $\mathcal{C}\subseteq\mathrm{mod}(A)$ if and only if $$\mathrm{dom.dim}(\mathrm{End}(M)^\mathrm{op})\geq d+1 \geq\mathrm{gl.dim}( \mathrm{End}(M)^\mathrm{op})$$ where $M$ is an additive generator of the subcategory 
$$\mathcal{C}=\mathrm{add}(\{\tau_d^i(DA)|i\geq 0\})\subseteq\mathrm{mod}(A).$$
\end{theorem}

For an $A$-module $M$, the annihilator of $M$ is the two-sided ideal $$\mathrm{ann}(M):=\{a\in A|Ma=0=aM\}.$$ The support of $M$, denoted $\mathrm{supp}(M)$, is defined to be the set of vertices such that for each vertex $i$, the module $S_i$ is contained in the composition series of $M$. 
For a module $M$, let 
\begin{align}
e_M&:=\sum_{i\notin\mathrm{supp}(M)}e_i\label{supp1},
\end{align}
and for a class of modules $\mathcal{T}\subseteq\mathcal{C}$ let 
\begin{align}
e_\mathcal{T}&:=\sum_{\{\not\exists M\in\mathcal{T}|i\in\mathrm{supp}(M)\}}e_i.\label{supp2}
\end{align}

An idempotent ideal $\langle e \rangle$ is \emph{$d$-idempotent} \cite[Section 1]{apt} if there are isomorphisms $\mathrm{Ext}^d_A(M,N)\cong \mathrm{Ext}^d_{A/\langle e \rangle}(M,N)$ for all $0\leq i\leq d$. Within higher Auslander-Reiten theory, $d$-idempotent ideals were heavily used in the construction of higher Nakayama algebras \cite[Lemma 1.20]{jk}. The singularity categories of higher Nakayama algebras were also described using idempotent ideals in \cite{mc2}. For the following result to make sense, we note that it is well-known that for a finite-dimensional algebra $A$ and injective $A$-module $I$, the $A/\langle e \rangle$-module $ \mathrm{Hom}_A(A/\langle e \rangle,I)$ is injective. Dually, for an injective $A$-module $P$, the $A/\langle e \rangle$-module $ P\otimes_AA/\langle e \rangle$ is projective.

\begin{prop}\cite[Proposition 1.1]{apt}\label{apt1}
 Let $A$ be a finite-dimensional algebra and let $e$ be an idempotent of $A$. Define the functor $F:=\mathrm{Hom}_A(A/\langle e \rangle,-)$ from $\mathrm{mod}(A)$ to $\mathrm{mod}(A/\langle e \rangle)$. Let $M$ be an $A$-module with minimal injective coresolution $$0\rightarrow M\rightarrow I_0\rightarrow \cdots\rightarrow I_d.$$ Then the following are equivalent:
\begin{enumerate}
\item The beginning of a minimal injective coresolution of $F(M)$ in $\mathrm{mod}(A/\langle e \rangle)$ is  $$0\rightarrow F(M)\rightarrow F(I_0)\rightarrow \cdots \rightarrow F(I_d).$$
\item $\mathrm{Ext}^i_A(A/\langle e \rangle,M)=0$ for all $1\leq i\leq d$.
\item For any $N\in\mathrm{mod}(A/\langle e \rangle)$ and for all $1\leq i\leq d$ there are isomorphisms $\mathrm{Ext}^i_A(N,M)\cong \mathrm{Ext}^i_{A/\langle e \rangle}(N,F(M))$ 
\end{enumerate}
\end{prop}
This leads to the following result.
\begin{prop}\cite[Proposition 1.3]{apt}\label{apt2}
 Let $A$ be a finite-dimensional algebra and let $e$ be an idempotent of $A$. Let $d\geq 1$ be a positive integer. Then the following conditions are equivalent:
\begin{enumerate}
\item The ideal $\langle e \rangle$ is $d$-idempotent.
\item $\mathrm{Ext}^i_A(A/\langle e \rangle, M)=0$ for all $A/\langle e \rangle$-modules $M$ and all $1\leq i\leq d$.
\item $\mathrm{Ext}^i_A(A/\langle e \rangle, I)=0$ for all injective $A/\langle e \rangle$-modules $I$ and all $1\leq i\leq d$.
\end{enumerate}
\end{prop}
This has the following implication for $d$-homological pairs.
\begin{cor}\label{adapt}
Let $(A,\mathcal{C})$ be a $d$-homological pair such that $\mathrm{gl.dim}(A)\leq d$. Suppose that $A/\langle e \rangle\in \mathcal{C}$, and that for any injective $A/\langle e \rangle$-module $I$, both $I\in\mathcal{C}$ and $\mathrm{Ext}^d_A(A/\langle e \rangle, I)=0$ . Then the following hold:
\begin{enumerate}
\item The ideal $\langle e \rangle$ is $(d+1)$-idempotent.
\item $\mathrm{gl.dim}(A/\langle e \rangle)\leq d$.
\end{enumerate}
\end{cor}

\begin{proof}
For any injective $A/\langle e \rangle$-module $I$, since both $A/\langle e \rangle\in \mathcal{C}$ and $I\in\mathcal{C}$ we have that $\mathrm{Ext}^i_A(A/\langle e \rangle, I)=0$ for all $0<i<d$. By assumption we also have that $\mathrm{Ext}^d_A(A/\langle e \rangle, I)=0$. Finally $\mathrm{gl.dim}(A)\leq d$ implies $\mathrm{Ext}^i_A(A/\langle e \rangle, I)=0$ for all $i>d$. By Proposition \ref{apt2}, this means the ideal $\langle e\rangle $ is $(d+1)$-idempotent. A consequence is that $\mathrm{Ext}^{d+1}_A(M,N)=\mathrm{Ext}^{d+1}_{A/\langle e \rangle}(M,N)=0$ for all $M,N\in\mathrm{mod}(A/\langle e \rangle)$. Therefore $\mathrm{gl.dim}(A/\langle e \rangle)\leq d$.
\end{proof}

Since $\langle e \rangle$ is $(d+1)$-idempotent, it would not be unreasonable to expect that $\mathcal{C}\cap\mathrm{mod}(A/\langle e\rangle)$ is a $d$-cluster-tilting subcategory of $\mathrm{mod}(A/\langle e\rangle)$. In fact the construction of higher Nakayama algebras relies on this property (see the usage of \cite[Lemma 1.20]{jk}. More generally, the following is true.

\begin{defn}\label{psupp}
 Let $(A,\mathcal{C})$ be a $d$-homological pair such that $\mathrm{gl.dim}(A)\leq d$. Then a subcategory $\mathcal{T}\subseteq \mathcal{C}$ is \emph{properly supported} and the idempotent $e_{\mathcal{T}}$ \emph{properly supporting} (for the idempotent $e_{\mathcal{T}}$ defined in equation \eqref{supp1}) if:
\begin{enumerate}
\item $\mathrm{Ext}^d_A(A/\langle e_\mathcal{T} \rangle, I)=0$ for all injective $A/\langle e_\mathcal{T} \rangle$-modules.
\item $\mathcal{C}\cap\mathrm{mod}(A/\langle e_\mathcal{T} \rangle)$ is a $d$-cluster-tilting subcategory of $\mathrm{mod}(A/\langle e_\mathcal{T}\rangle)$ such that \begin{align*}
\mathcal{C}\cap\mathrm{mod}(A/\langle e_\mathcal{T} \rangle)&\cong\{\mathrm{Hom}_A(A/\langle e\rangle, M)|M\in\mathcal{C}\}\\
&\cong\{M\otimes_A A/\langle e\rangle|M\in\mathcal{C}\}.
\end{align*}\label{suppg}
\end{enumerate}
If for some $T\in\mathcal{C}$ the subcategory $\mathrm{add}(T)\subseteq \mathcal{C}$ is properly supported, then we say that $T$ is \emph{properly supported}. An idempotent $e$ is \emph{left properly supporting} if $e$ is properly supporting and $I\in\mathrm{add}(DA\otimes_A A/\langle e \rangle)$ for every injective $A/\langle e \rangle$ module $I$. Dually, a properly-supporting idempotent $e$ is \emph{right properly supporting} if there is an identification $\mathrm{add}(\mathrm{Hom}_A(A/\langle e \rangle,A)) \cong \mathrm{add}(\mathrm A/\langle e \rangle)$.
\end{defn}

Recall that for a finite-dimensional algebra $A$ and any two morphisms \\ $f:X\rightarrow M$, $g:X\rightarrow N$ between objects $M, N, X\in\mathrm{mod}(A)$, there is a \emph{pushout of $f$ and $g$} consisting of an object $P$ and morphisms $f^\prime:M\rightarrow P$, $g^\prime:N\rightarrow P$ such that $f^\prime\circ f\cong g^\prime\circ g$ and $P$ is universal with this property: for any $P_1$ and morphisms $f_1:M\rightarrow P_1$, $g_1:N\rightarrow P_1$ such that $f_1\circ f\cong g_1\circ g$, then $f_1$ factors through $f^\prime$ and $g_1$ factors through $g^\prime$.
 One property of the pushout is that there is an exact sequence $$X\rightarrow M\oplus N\rightarrow P.$$ More generally, in a functorially-finite subcategory $\mathcal{C}\subseteq \mathrm{mod}(A)$ and any two morphisms $f:X\rightarrow M$, $g:X\rightarrow N$ between objects $M, N, X\in \mathcal{C}$, the property of being preenveloping implies that there is an object $P\in \mathcal{C}$ and morphisms $f^\prime:M\rightarrow P$, $g^\prime:N\rightarrow P$ such that $f^\prime\circ f\cong g^\prime\circ g$, that $P$ is universal with this property and there is an exact sequence $$X\rightarrow M\oplus N\rightarrow P.$$

Returning to the category $\mathrm{mod}(A)$, if $P$ is the pushout of two morphisms $f$ and $g$ such that $f$ is an injective morphism, then there is a commutative diagram 
$$\begin{tikzcd}
0\arrow{r}&X\arrow{r}{f}\arrow{d}{g}&M\arrow{r}\arrow{d}&C\arrow{r}\arrow[equals]{d}&0\\
0\arrow{r}&N\arrow{r}&P\arrow{r}&C\arrow{r}&0
\end{tikzcd}$$ for some $C\in \mathrm{mod}(A)$. This concept (and the dual notion of a pullback) was generalised to $d$-pushout and $d$-pullback diagrams in \cite{jasso}, see also the survey article \cite{jkv}. First, an exact sequence is \emph{$d$-exact} if it can be written in the form $$0\rightarrow M_0\rightarrow M_1\rightarrow \cdots \rightarrow M_{d+1}\rightarrow 0.$$ The result we need is the following:

\begin{prop}\cite[Proposition 3.8]{jasso}\label{jasso}
Let $(A,\mathcal{C})$ be a $d$-homological pair. For any $d$-exact sequence in $\mathcal{C}$
$$0\rightarrow X_0\rightarrow X_1\rightarrow\cdots \rightarrow X_{d+1}\rightarrow 0$$ and any morphism $f:X_0\rightarrow Y_0$ there exists a commutative diagram in $\mathcal{C}$:

$$\begin{tikzcd}
0\arrow{r}&X_0\arrow{r}\arrow{d}{f}&X_1\arrow{r}\arrow{d}&\cdots \arrow{r}&X_{d}\arrow{d}\arrow{r}&X_{d+1}\arrow{r}\arrow[equals]{d}&0\\
0\arrow{r}&Y_0\arrow{r}&Y_1\arrow{r}&\cdots \arrow{r}&Y_{d}\arrow{r}&X_{d+1}\arrow{r}&0
\end{tikzcd}$$
such that there is an induced $d$-exact sequence
$$0\rightarrow X_0\rightarrow X_1\oplus Y_0\rightarrow X_2\oplus Y_1\rightarrow \cdots\rightarrow X_d\oplus Y_{d-1}\rightarrow Y_d\rightarrow 0$$
\end{prop}

The commutative diagram 
$$\begin{tikzcd}
X_0\arrow{r}\arrow{d}{f}&X_1\arrow{r}\arrow{d}&\cdots \arrow{r}&X_{d}\arrow{d}\\
Y_0\arrow{r}&Y_1\arrow{r}&\cdots \arrow{r}&Y_{d}
\end{tikzcd}$$
is a \emph{$d$-pushout diagram}. Dually, there is the notion of a \emph{$d$-pullback diagram}. For an algebra $A$ with global dimension $d$ any any two $A$-modules $M$ and $N$, then as in \cite[Theorem 2.3.1]{iy1} there is an isomorphism $$\mathrm{Hom}_A(M,\tau_d(N))\cong \mathrm{Ext}_A^d(N,M).$$ 
For a $d$-representation-finite algebra $A$ with $d$-cluster-tilting subcategory $\mathcal{C}$, there is an alternative way of phrasing this isomorphism. For any $M,N\in\mathcal{C}$ and any non-zero morphism $f\in\mathrm{Hom}(M,\tau_d(N))$, then a $d$-pushout diagram induces the following commutative diagram:

$$\begin{tikzcd}
0\arrow{r}&M\arrow{r}\arrow{d}{f}&X_1\arrow{r}\arrow{d}&\cdots \arrow{r}&X_d \arrow{r}\arrow{d}&N\arrow[equals]{d}\arrow{r}&0\\
0\arrow{r}&\tau_d(N)\arrow{r}&Y_1\arrow{r}&\cdots \arrow{r}&Y_{d}\arrow{r}&N\arrow{r}&0
\end{tikzcd}$$
 So Proposition \ref{jasso} implies that for any modules $M,N\in\mathcal{C}$, then any $d$-exact sequence  $$0\rightarrow M\rightarrow E_1\rightarrow\cdots\rightarrow E_d\rightarrow N\rightarrow 0$$ is Yoneda equivalent to a $d$-exact sequence 
  $$0\rightarrow N\rightarrow X_1 \rightarrow \cdots \rightarrow X_d\rightarrow M\rightarrow 0$$ such that $X_1,\ldots, X_d\in \mathcal{C}$.

\section{Almost directed subcategories and resonance diagrams}
We desire a version of directedness in order to be able to generalise torsion classes.

\begin{defn}\label{almost} 
A $d$-homological pair $(A,\mathcal{C})$ is \emph{almost directed} if the following conditions are satisfied.
\begin{enumerate}
\item The algebra $A$ is given by a quiver with relations whose commutativity relations and zero relations are generated by paths of length two. \label{nak0}
\item For any indecomposable modules $M,N\in\mathcal{C}$ and any integer $i\geq 1$, then $\mathrm{dim}(\mathrm{Ext}_A^i(M,N))\leq 1$.\label{nak1}
\item For any $M\in\mathcal{C}$, there exists a left properly-supporting idempotent $e$ such that either $M$ is projective as an $A/\langle e \rangle$-module, or for any \\ $N\in\mathcal{C}\cap\mathrm{mod}(A/\langle e \rangle)$ then $\mathrm{Hom}_A(M,N)\ne 0$ implies that $N$ is injective as an $A$-module. \label{nak2}
\item For any $M\in\mathcal{C}$, there exists a right properly-supporting idempotent $e$ such that either $M$ is injective as an $A/\langle e \rangle$-module, or for any \\ $N\in\mathcal{C}\cap\mathrm{mod}(A/\langle e \rangle)$ then $\mathrm{Hom}_A(N,M)\ne 0$ implies that $N$ is projective as an $A$-module. \label{nak6}
\end{enumerate}
\end{defn}

Observe that if a $d$-homological pair $(A,\mathcal{C})$ is almost directed, then for any properly-supporting idempotent $e$, then the $d$-homological pair $(A/\langle e\rangle,\mathcal{C} \cap\mathrm{mod}(A/\langle e\rangle))$ is almost directed. A wide range of algebras with $d$-cluster-tilting subcategories that do not satisfy property (\ref{nak0}) can be found in \cite{vaso2}.  

\begin{prop}\label{resonance}
Let $(A,\mathcal{C})$ be an almost-directed $d$-homological pair such that $\mathrm{gl.dim}(A)\leq d$. Suppose that 
$$\phi_M:0\rightarrow M_0\rightarrow M_1\rightarrow \cdots \rightarrow M_{d+1}\rightarrow 0$$ 
$$\phi_N:0\rightarrow N_0\rightarrow N_1\rightarrow \cdots \rightarrow N_{d+1}\rightarrow 0$$ are $d$-exact sequences in $\mathcal{C}$ such that $M_i=N_j$ for some $0\leq i\leq j\leq d+1$.  Then there exists a commutative diagram containing $\phi_M$, $\phi_N$ and $(d-1)$ many other $d$-exact sequences $$\phi_k:X_{0,k}\rightarrow X_{1,k}\rightarrow \cdots \rightarrow X_{d+1,k}\rightarrow 0$$ where $1\leq k\leq d-1$ and such that each module $X_{l,k}$ is a term in either $\phi_M$, $\phi_N$ or $\phi_m$ for some $1\leq m\leq d-1$ (where $l\ne m$) and this is distinct for each $0\leq l\leq d+1$.
\end{prop}

\begin{proof}
Because $\mathcal{C}$ is functorially finite, we may take the ``pullback" $P_{i-1}$ of $M_i=N_j$ along the morphisms $M_{i-1}\rightarrow M_i$ and $h:N_{j-1}\rightarrow N_j$ to obtain morphisms \\ $f_{i-1}:P_{i-1}\rightarrow M_{i-1}$, $g_{i-1}:P_{i-1}\rightarrow N_{j-1}$. Now take the ``pullback" $P_{i-2}$ of \\ $M_{i-1}\rightarrow M_i$ and $f_{i-1}$ to obtain morphisms $f_{i-2}:P_{i-2}\rightarrow M_{i-2}$, $g_{i-2}:P_{i-2}\rightarrow P_{i-1}$. Since the composition of morphisms $h\circ g_{i-1}\circ g_{i-2}$ factors through $f_{i-1}\circ f_{i-2}$, we must have that $h\circ g_{i-1}\circ g_{i-2}=0$. Since all zero relations are of length at most two, this means the composition $g_{i-1}\circ g_{i-2}=0$. Now, as the sequences $$P_{i-2}\rightarrow P_{i-1}\oplus M_{i-2}\rightarrow M_{i-1}$$ 
$$P_{i-1}\rightarrow N_{j-1}\oplus M_{i-1}\rightarrow M_{i}$$ are exact and $g_{i-1}\circ g_{i-2}=0$ by uniqueness of $P_{i-1}$, we must have that $$P_{i-2}\rightarrow P_{i-1}\rightarrow N_{j-1}$$ is exact. Continue this process through $\phi_M$, to $\phi_N$ and dually.  Inductively we obtain more $d$-exact sequences, and hence every induced module must be contained in two of the induced $d$-exact sequences. 
\end{proof}

A diagram as in Proposition \ref{resonance} will be called a \emph{resonance diagram}.
Fix an integer $n$ and consider the $d$-Auslander algebra $A$ of linearly oriented type $A_n$.  There is a $d$-cluster-tilting subcategory $\mathcal{C}\subseteq \mathrm{mod}(A)$ that can be described by Theorem \ref{artheory} and that has also been described further through results on higher Nakayama algebra in \cite{jk}. It follows from these descriptions that $\mathcal{C}$ is almost directed.

\begin{eg}\label{eg1}
Let $A$ be the Auslander algebra of linearly oriented type $A_3$

	$$
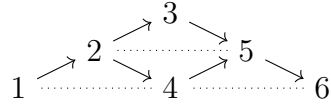
\begin{tikzpicture}[xscale=5,yscale=2.5]
	\node(xa) at (-2,1){$1$};
	\node(xb) at (-1.6,1){$4$};
	\node(xc) at (-1.2,1){$6$};
	
	\node(xab) at (-1.8,1.2){$2$};
	\node(xac) at (-1.4,1.2){$5$};
	
	\node(xbb) at (-1.6,1.4){$3$};
	
	\draw[->](xa) edge(xab);
	\draw[->](xb) edge(xac);
	
	\draw[->](xbb) edge(xac);
		
	\draw[->](xab) edge(xb);
	\draw[->](xac) edge(xc);
	
	\draw[->](xab) edge(xbb);
	
	\draw[-,dotted](xc) edge(xb);
	\draw[-,dotted](xb) edge(xa);
	\draw[-,dotted](xac) edge(xab);

	\end{tikzpicture}$$
	and let $\mathcal{C}$ be the cluster-tilting subcategory $$\begin{matrix} 1\\2\\3\end{matrix},\ \begin{matrix} 1\\2\end{matrix},\ \begin{matrix} 1\end{matrix},\ \begin{matrix} &2&\\3&&4\\&5&\end{matrix},\ \begin{matrix} 2\\4\end{matrix},\ \begin{matrix} 3\\5\\6\end{matrix},\ \begin{matrix} 4\\5\end{matrix},\ \begin{matrix} 4\end{matrix},\ \begin{matrix} 5\\6\end{matrix},\ \begin{matrix} 6\end{matrix}.$$
	It may be easily seen that $(A,\mathcal{C})$ is an almost-directed $d$-homological pair.
	\end{eg}
	
For algebras of linearly oriented type $A_n$, we may give a much more explicit description of their resonance diagrams.

\begin{eg}\label{aresonance}
Let $A$ be a $d$-Auslander algebra of linearly oriented type $A_n$ with unique $d$-cluster-tilting subcategory $\mathcal{C}\subseteq\mathrm{mod}(A)$. Suppose that 
$$\phi_M:0\rightarrow M_0\rightarrow M_1\rightarrow \cdots \rightarrow M_{d+1}\rightarrow 0$$ 
$$\phi_N:0\rightarrow N_0\rightarrow N_1\rightarrow \cdots \rightarrow N_{d+1}\rightarrow 0$$ are $d$-exact sequences in $\mathcal{C}$ such that $M_i=N_j$ for some $0\leq i\leq j\leq d+1$. 
\begin{figure}
\caption{Resonance diagram for a $d$-Auslander algebra of linearly oriented type $A_n$}
$$\begin{tikzcd}[row sep=small, column sep=tiny]\label{resii}
X_{0,0}\arrow[]{d}\arrow{dr}\\
X_{0,1}\arrow{d}\arrow{r}&X_{1,1}\arrow{d}\arrow{dr}\\
\vdots \arrow{d}&\vdots\arrow{d}&\ddots\arrow{dr}\\
X_{0,i-1} \arrow{r}\arrow{d}&X_{1,i-1}\arrow{d}\arrow{r}&\cdots\arrow{r}&X_{i-1,i-1}\arrow{dr}\\
\vdots&\vdots &&&X_{i,i}\arrow{d}\arrow{dr}\\
\vdots\arrow{d}&\vdots\arrow{d}&&&\vdots\arrow{d}&\ddots\arrow{dr}\\
X_{0,j} \arrow{r}\arrow{d}&X_{1,j}\arrow{d}\arrow{r}&\cdots&\cdots\arrow{r}&X_{i,j}\arrow{d}\arrow{r}&\cdots\arrow{r}&X_{j,j}\arrow{dr}\\
\vdots&\vdots &&&\vdots&&&X_{j+1,j+1}\arrow{d}\arrow{dr}\\
\vdots\arrow{d}&\vdots\arrow{d}&&&\vdots\arrow{d}&&&\vdots\arrow{d}&\ddots\arrow{dr}	\\
X_{0,d+1} \arrow{r}&X_{1,d+1}\arrow{r}&\cdots&\cdots\arrow{r}&X_{i,d+1}\arrow{r}&\cdots&\cdots\arrow{r}&X_{j+1,d+1}\arrow{r}&\cdots\arrow{r}&X_{d+1,d+1}
\end{tikzcd}$$ 
\end{figure}
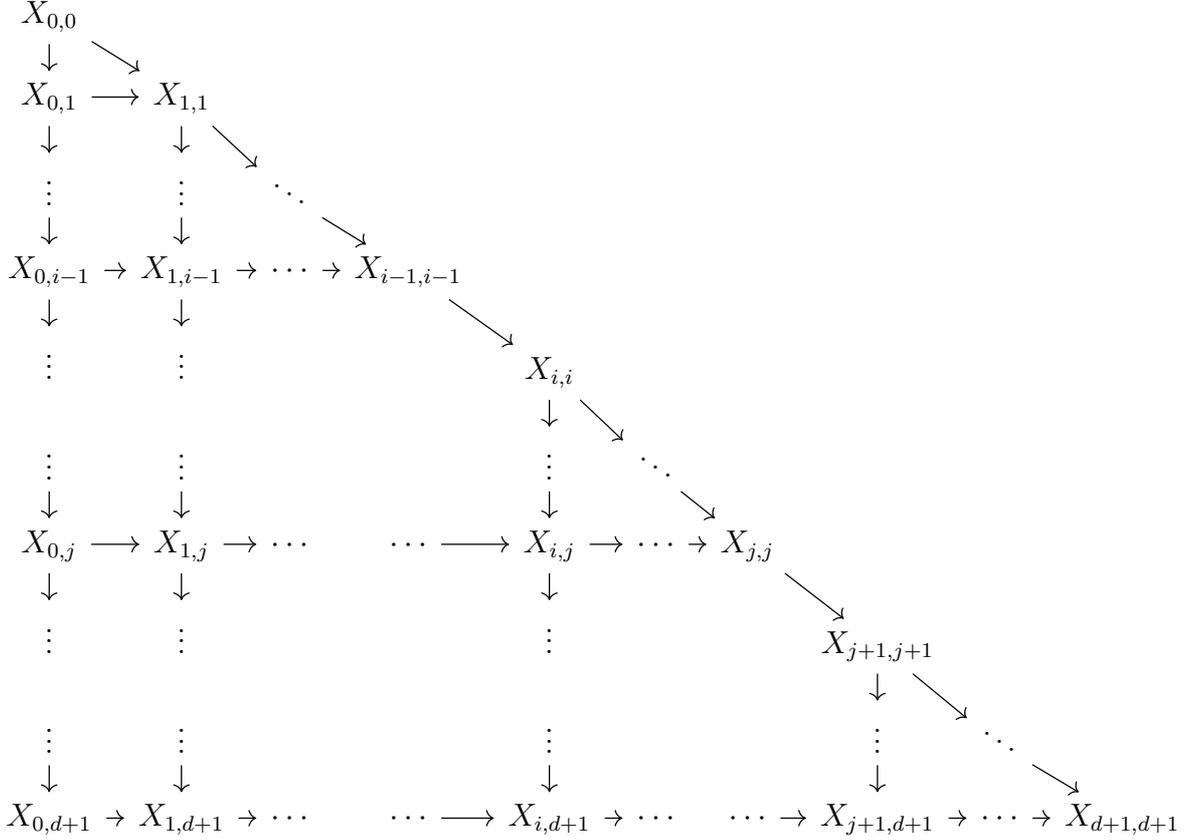  Then there is a commutative diagram as in Figure \ref{resii} 
such that 
\begin{itemize}
\item For each $0\leq k\leq d+1$ the sequence $$\phi_k:0\rightarrow X_{0,k}\rightarrow X_{1,k}\rightarrow  \cdots \rightarrow X_{k,k}\rightarrow X_{k+1,k+1}\rightarrow X_{k+1,k+2}\rightarrow \cdots \rightarrow X_{k+1,d+1}\rightarrow 0$$ is $d$-exact. \label{y1}
\item Either $\phi_M=\phi_{i-1}$ and $\phi_N=\phi_j$ or $\phi_N=\phi_{i-1}$ and $\phi_M=\phi_i$. \label{y3}
\end{itemize}
\end{eg}

\begin{remark}\label{resin}
Observe that for each $0\leq k \leq d+1$, by construction the sequence $$\phi_{k}^\prime:0\rightarrow X_{0,k}\rightarrow X_{1,k}\oplus X_{0,k+1}\rightarrow \cdots \rightarrow X_{k-1,d+1}\oplus X_{k,d}\rightarrow X_{k,d+1}\rightarrow 0$$ is $d$-exact. In fact this $d$-exact sequence is obtained via a $d$-pushout diagram of $\phi_{k-1}$ along the morphism $X_{0,k-1}\rightarrow X_{0,k}$. Likewise, it may also be obtained via a $d$-pullback diagram of $\phi_k$ along the morphism $X_{k,d+1}\rightarrow X_{k+1,d+1}$. 
\end{remark}

For an arbitrary resonance diagram we call $d$-exact sequences that arise from either a $d$-pushout diagram or a $d$-pullback diagram, or both, as \emph{intermediate $d$-exact sequences}.

\section{Support $d$-tilting modules} 

The main objective of this paper is to understand support $d$-tilting theory. Recall that an $A$-module $T$ is a \emph{pre-$d$-tilting module} \cite{ha}, \cite{miya} if:

\begin{itemize}
\item $\mathrm{proj.dim}(T)\leq d$.
\item $\mathrm{Ext}^i_A(T,T)=0$ for all $0<i\leq d$.
\end{itemize}
Then $T$ is in addition \emph{$d$-tilting} if there exists an exact sequence $$0\rightarrow A\rightarrow T_0\rightarrow T_1\rightarrow \cdots \rightarrow T_d\rightarrow 0$$ 
where $T_0,\ldots, T_d\in \mathrm{add}(T)$. 
Dually, an $A$-module $C$ is a \emph{$d$-cotilting module} if:
\begin{itemize}
\item $\mathrm{inj.dim}(C)\leq d$.
\item $\mathrm{Ext}^i_A(C,C)=0$ for all $0<i\leq d$.
\item There exists an exact sequence $$0\rightarrow C_d\rightarrow \cdots \rightarrow C_1\rightarrow C_0 \rightarrow DA\rightarrow 0$$ such that $C_0,\ldots, C_d\in\mathrm{add}(C)$.
\end{itemize} 

The notion of support-$d$-tilting follows naturally, but we will be primarily interested in modules that are properly supported.

\begin{defn}\label{suppt}
 Let $(A,\mathcal{C})$ be a $d$-homological pair such that $\mathrm{gl.dim}(A)\leq d$ and that $\mathcal{C}$ is almost directed. Then an $A$-module $T\in \mathcal{C}$ is \emph{support-$d$-tilting} if
	\begin{enumerate}
	\item  $T$ is $d$-tilting as an $A/\mathrm{ann}(T)$-module.
	\item $\mathrm{ann}(T)=\langle e_T\rangle$,
	\end{enumerate}
where $e_T$ is the idempotent defined in equation \eqref{supp1}. If $T$ is properly supported, then we say that $T$ is a \emph{proper support-$d$-tilting module}. 
Dually an $A$-module $C\in \mathcal{C}$ is \emph{proper support-$d$-cotilting} if
 \begin{itemize}
 \item $C$ is $d$-cotilting as an $A/\mathrm{ann}(C)$-module.
 \item $\mathrm{ann}(C)=\langle e_C\rangle$,
 \item $C$ is properly supported.
\end{itemize} then $C$ is a \emph{proper support-$d$-cotilting module}. 
\end{defn}

\begin{lm}[Happel \cite{ha}, as found in Lemma 3.5 of \cite{io}]\label{happel}
Let $A$ be a finite-dimensional algebra satisfying  $\mathrm{gl.dim}(A)\leq d$. Let $T$ be a $d$-tilting $A$-module. Assume that $M\in \mathrm{mod}(A)$ satisfies $\mathrm{Ext}_A^i(T,M)=0$ for all $i>0$. Then there exists an exact sequence $$0\rightarrow T_d\rightarrow \cdots \rightarrow T_1\rightarrow T_0\rightarrow M\rightarrow 0$$ such that $T_j\in\mathrm{add}(T)$ for all $0\leq j\leq d$.
\end{lm}

Consider a $d$-tilting $A$-module $T$. Since every injective module $I$ satisfies \\ $\mathrm{Ext}_A^i(T,I)=0$ for all $i>0$ we have the following corollary.

\begin{cor}
Let $(A,\mathcal{C})$ be a $d$-homological pair such that $\mathrm{gl.dim}(A)\leq d$. If $T$ is a support-$d$-tilting module, then every (proper) support-$d$-tilting module is also (proper) support-$d$-cotilting. 
\end{cor}

It is not true in general that every pre-$d$-tilting module is a summand of a $d$-tilting module - in other words Bongartz' Lemma cannot be fully realised in this context. For a pre-$d$-tilting $A$-module $T$ such that there is no indecomposable $A$-module $M$ such that $T\oplus M$ is also pre-$d$-tilting, we say that $T$ is \emph{maximal pre-$d$-tilting}. 
The following result is useful for determining support-$d$-tilting modules. The notion of \emph{maximal support-pre-$d$-tilting} can be defined similarly.

\begin{theorem}\cite[Theorem 1.19]{miya}
For a finite-dimensional algebra $A$, if $M$ is a $d$-tilting module, then $|M|=|A|$, where $|M|$ is the number of indecomposable summands of $M$. 
\end{theorem}

From this result, we would hope that an alternative definition of a support-$d$-tilting $A$-module to be a module $T$ such that $T$ is a tilting $A/\langle e_T\rangle$-module and $|T|=|A/\langle e_T\rangle|$.

\section{Strong torsion classes}

\begin{defn}\label{torsion}
Let $(A,\mathcal{C})$ be a $d$-homological pair. Then a full subcategory $\mathcal{T}\subseteq \mathcal{C}$ is a \emph{$d$-strong torsion class in $\mathcal{C}$} if
\begin{enumerate}[(T1)]
\item For any $T\in\mathcal{T}$, $M\in \mathcal{C}$ and surjective morphism $T\twoheadrightarrow M$, then $M\in \mathcal{T}$.\label{torsi}
\item For each $d$-exact sequence in $\mathcal{C}$ 
$$0\rightarrow M_0\rightarrow M_1\rightarrow M_2\rightarrow \cdots \rightarrow M_d\rightarrow M_{d+1} \rightarrow 0$$ such that $M_{i-1},M_{i+1}\in\mathcal{T}$ for some $1\leq i\leq  d$, then also $M_i\in \mathcal{T}$\label{torsii}
\end{enumerate}
 Dually, a full subcategory $\mathcal{F}\subseteq \mathcal{C}$ is a \emph{$d$-strong torsion-free class in $\mathcal{C}$} if 
 
\begin{enumerate}[(C1)]
\item For any $F\in\mathcal{F}$, $M\in \mathcal{C}$ and injective morphism $M\hookrightarrow F$, then $M\in \mathcal{F}$.\label{coti}
\item For each $d$-exact sequence in $\mathcal{C}$ 
$$0\rightarrow M_0\rightarrow M_1\rightarrow M_2\rightarrow \cdots \rightarrow M_d\rightarrow M_{d+1} \rightarrow 0$$ such that $M_{i-1},M_{i+1}\in\mathcal{F}$ for some $1\leq i\leq  d$, then also $M_i\in \mathcal{F}$\label{cotii}
\end{enumerate}
\end{defn}

An alternative version of a torsion class for higher homological algebra was introduced in \cite{jorg}. It would be interesting to find out how closely the two notions are related in general. We require two preliminary results.

\begin{lm}\label{chevelle}
Let $(A,\mathcal{C})$ be an almost-directed $d$-homological pair such that \\ $\mathrm{gl.dim}(A)\leq d$. Let $T$ be a $d$-tilting $A$-module and suppose that
$$P_0\rightarrow P_1\rightarrow \cdots \rightarrow P_l$$ be an exact sequence such that $P_i$ is an indecomposable projective for each \\$0\leq i\leq l$.  If $P_0, P_l\in\mathrm{add}(T)$, then $P_i\in\mathrm{add}(T)$ for all $0\leq i\leq l$.
\end{lm}

\begin{proof}
We may assume that for each $1\leq i\leq l-1$, then $P_i\notin\mathrm{add}(T)$. Then, since $T$ is $d$-tilting,  for each $1\leq i\leq l-1$ there exists a $d$-exact sequence $$\phi_i:0\rightarrow P_i\rightarrow T_{0,i}\rightarrow T_{1,i}\rightarrow \cdots \rightarrow T_{d,i}\rightarrow 0$$ such that $T_{j,i}\in\mathrm{add}(T)$ for all $0\leq j\leq d$. As $\mathrm{Ext}^d_A(T,P_l)=0$, we must have that $P_{l-1}\rightarrow P_l$ factors through $T_{0,l-1}$, as well as that $\mathrm{Ext}^d_A(T_{d,l-1},P_{l-2})\ne 0$. So we may assume that $T_{d,l-1}\cong T_{d,l-2}$ and that the $d$-exact sequences $\phi_l$ and $\phi_{l-1}$ are contained in some resonance diagram. However, by taking an intermediate $d$-exact sequence we must have that $\mathrm{Ext}^d_A(T_{d-1,l-1},P_{l-2})\ne 0$. This further means that $\mathrm{Ext}^d_A(T_{d-1,l-1},P_{l-3})\ne 0$, in other words that $T_{d,l-3}\cong T_{d-1,l-1}$. Continuing this argument further implies that $P_0\notin\mathrm{add}(T)$, a contradiction.
\end{proof}

\begin{lm}\label{skel2}
Let $(A,\mathcal{C})$ be an almost-directed $d$-homological pair such that \\$\mathrm{gl.dim}(A)\leq d$. Let $T$ be a $d$-tilting $A$-module. Then for any left properly-supporting idempotent $e$, the $A/\langle e \rangle$-module $T\otimes_A A/\langle e \rangle$ is $d$-tilting. Dually for any right properly-supporting idempotent $e$,  the $A/\langle e \rangle$-module $\mathrm{Hom}_A(A/\langle e \rangle, T)$ is $d$-tilting.
\end{lm}

\begin{proof}
Suppose that $e$ is left properly-supporting. Then, by assumption, any injective $A$-module is of the form $G(I)$ for some $I\in \mathrm{add}(DA)$, where we use the notation $G:-\otimes_A A/\langle e \rangle$. For any injective $A$-module $I$ there is a $d$-exact sequence
$$\phi:0\rightarrow T_d\rightarrow \cdots \rightarrow T_0\rightarrow I\rightarrow 0$$
such that $T_i\in\mathrm{add}(T)$ for all $0\leq i\leq d$.  Suppose there exists a minimal $0\leq j\leq d$ such that $T_j$ is not projective. Else apply the functor $G$ to $\phi$ and this gives a $d$-exact sequence by Proposition \ref{apt1}. Consider the morphism $T_{j}\rightarrow T_{j-1}$. Then there is a non-zero morphism between the projective cover of $T_j$ and $T_{j-1}$.  This must factor though the projective cover of $T_{j-1}$ and form a commutative square of the form found in a resonance diagram.  So complete this square to a resonance diagram between $\phi$ and the projective resolution of $T_j$:
$$0\rightarrow P_d\rightarrow \cdots \rightarrow P_0\rightarrow T_j\rightarrow 0.$$ This resonance diagram contains the projective resolutions of all $T_i$ for $0\leq i\leq d$. The functor $G$ must be exact for any of these projective resolutions by Proposition \ref{apt1}, and thus in this case also $G(\phi)$ must be $d$-exact. Finally, since $e$ is properly supporting, we have that $\langle e \rangle$ is $(d+1)$-idempotent by Corollary \ref{adapt}, and hence $\mathrm{Ext}^d_{A/\langle e \rangle}(G(T),G(T))=0$. So $G(T)$ is a $d$-tilting $A/\langle e \rangle$-module. Dually, if $e$ is right properly supporting then $\mathrm{Hom}_A( A/\langle e \rangle,T)$ is a $d$-tilting $A/\langle e \rangle$-module.
 \end{proof}

The following theorem generalises the classical behaviour of torsion classes. Specifically, we generalise Propositions VI.1.4 and  VI.1.5 from the book \cite{ass}.

\begin{theorem}\label{elso}
Let $(A,\mathcal{C})$ be an almost-directed $d$-homological pair such that \\ $\mathrm{gl.dim}(A)\leq d$. For any proper support-$d$-tilting $A$-module $T$, let $\mathcal{T}:=\mathrm{Fac}(T)\cap \mathcal{C}$. Then the following hold:
\begin{enumerate}
\item $\mathcal{T}$ is a $d$-strong torsion class in $\mathcal{C}$.\label{t2}
\item For all $M\in \mathcal{C}$, there is a module $F_M\in\mathrm{mod}(A)$ and an exact sequence
$$0\rightarrow T_1\rightarrow T_2\rightarrow \cdots\rightarrow T_d\rightarrow M \rightarrow F_M\rightarrow 0$$
such that $T_i\in\mathcal{T}$ for all $1\leq i\leq d$ and $$\mathcal{T}=\{T\in\mathcal{C}|\mathrm{Hom}_A(T,F_M)=0\ \forall M\in\mathcal{C}\}.$$ \label{t3}
\end{enumerate}
\end{theorem}
\begin{proof}
To prove part (\ref{t3}), suppose that $T$ is a support-$d$-tilting algebra. For any indecomposable module $M\in\mathcal{C}$, let $t(M)$ be the trace of $\mathcal{T}$ in $M$; the sum of the images of all $A$-homomorphisms from modules in $\mathcal{T}$ to $M$. So take $f:T^\prime \rightarrow M$ to be the morphism from a module $T^\prime\in\mathrm{add}(T)$ to $M$ such that $\mathrm{im}(f)=t(M)$.

By Lemma \ref{happel} every injective $A$-module is in $\mathcal{T}$. So it follows by assumption (\ref{nak2}) in Definition \ref{almost} that there exists a left properly-supporting idempotent $e$ such that $T^\prime$ is projective as a $A/\langle e \rangle$-module. Let $T^\prime=:P_0$ and $G:=-\otimes_A A/\langle e\rangle$. By Lemma \ref{skel2}, $G(T)$ is a $d$-tilting $A/\langle e \rangle$-module, and we may obtain an exact sequence
$$0\rightarrow P_{n}\rightarrow P_{n-1} \rightarrow \cdots \rightarrow P_0\rightarrow M$$ such that $P_0,P_1,\ldots,P_n$ are projective as $A/\langle e \rangle$-modules. Let $M^\prime$ be the preenvelope of $\mathrm{coker}(P_1\rightarrow P_0)$ in $\mathrm{mod}(A/\langle e \rangle)\cap \mathcal{C}$, so $M^\prime$ is a submodule of $M$. 
We claim that there exists some $1\leq i\leq d$ and a $d$-exact sequence
$$0\rightarrow P_n \rightarrow \cdots\rightarrow P_0 \oplus N_{i-1} \rightarrow M^\prime\oplus N_i \rightarrow \cdots \rightarrow N_{d+1}\rightarrow 0.$$ 
This is clear if $P_{n-1}$  is indecomposable. Otherwise there must be some decomposition $P_{n-1}=Q\oplus Q^\prime$ such that  cokernels of $P_n\rightarrow Q$ and $P_n\rightarrow Q^\prime$ form a resonance diagram where there is a cokernel of $P_n\rightarrow P_{n-1}$ as an intermediate $d$-exact sequence. Therefore such a $d$-exact sequence as claimed exists.
Since $M\notin \mathcal{T}$, this means that $n<d$. 

Now assume that $P_1$ is not indecomposable, and choose a resonance diagram so that there is a $d$-exact sequence $$0\rightarrow X_0\rightarrow X_1\rightarrow\cdots \rightarrow X_{d+1}\rightarrow 0$$ such that there is some $0\leq j\leq d-1$ where $X_j$ is an indecomposable summand of $P_1$ and $X_{j+1}$ is an indecomposable summand of $G(T)$. We also have that $X_1,\ldots, X_{j+1}$ are indecomposable projective $A/\langle e \rangle$-modules. It follows that the morphism $X_{j+1}\rightarrow M$ factors through $X_{j+2}\oplus Y$ for some $Y\in\mathcal{C}$. Now assume that there is some $0\leq k\leq j$ such that $X_k\notin\mathrm{add}(G(T))$ and $X_{k+1}\in\mathrm{add}(G(T))$. By Lemma \ref{chevelle}, this means that $X_{k+1}, X_{k+2},\ldots, X_{j+1}\in\mathrm{add}(G(T))$. Moreover, in the $d$-exact sequence $$\phi:0\rightarrow X_k\rightarrow T_0\rightarrow\cdots\rightarrow T_d\rightarrow 0,$$ at least one of $Y\in\mathrm{add}(G(T))$ or $X_{j+2}\in\mathrm{add}(G(T))$. Since the functor $G$ is exact on $\phi$ (as in the proof of Lemma \ref{skel2}), this implies that $\mathrm{Im}(P_0\rightarrow M)\ne t(M)$, a contradiction. 

So either $P_0, P_1,\ldots, P_n\in\mathrm{add}(G(T))$ or they are all indecomposable projective $A/\langle e \rangle$-modules. As in the proof of Lemma \ref{chevelle}, there is a resonance diagram containing the $\mathrm{add}(G(T))$-coresolutions of each of these indecomposable projective modules. From this diagram, we may obtain a sequence $$0\rightarrow G(T_1)\rightarrow G(T_2)\rightarrow \cdots\rightarrow G(T_d)\rightarrow M \rightarrow M/t(M)\rightarrow 0.$$ Since $G(T_d)\cong T_d$,  and $G(T_i)$ is a factor module of $T_i$ for all $1\leq i\leq d$, this completes the proof. 

Note that $\mathrm{Hom}_A(T,F_M)=0$ for any $T\in\mathcal{T}$ and $M\in\mathcal{C}$ by definition, and $\mathrm{Hom}_A(M,F_M)= 0$ if any only if $F_M=0$, which means $M\in\mathcal{T}$. Part (\ref{t2}) follows from the exactness properties of $\mathrm{Ext}^d_A(-,F_M)$ for any $M\in\mathcal{C}$.
\end{proof}

For completeness' sake, we state the dual result.
\begin{theorem}
Let $(A,\mathcal{C})$ be an almost-directed $d$-homological pair such that\\ $\mathrm{gl.dim}(A)\leq d$. For any proper support-$d$-cotilting $A$-module $F$, let \\$\mathcal{F}:=\mathrm{Sub}(F)\cap \mathcal{C}$. Then the following hold:
\begin{enumerate}
\item $\mathcal{F}$ is a $d$-strong torsion-free class.
\item For all $M\in \mathcal{C}$, there is a module $T_M\in\mathrm{mod}(A)$ and an exact sequence
$$0\rightarrow T_M\rightarrow M\rightarrow F_0 \rightarrow F_1\rightarrow \cdots\rightarrow F_d\rightarrow 0$$
such that $F_i\in\mathcal{F}$ for all $1\leq i\leq d$ and $$\mathcal{F}=\{F\in\mathcal{C}|\mathrm{Hom}_A(T_M,F)=0\ \forall M\in\mathcal{C}\}.$$ 
\end{enumerate}
\end{theorem}

Theorem \ref{elso} does not provide a classification of strong $d$-torsion classes. However we conjecture that the following is true:

\begin{conj}
Let $(A,\mathcal{C})$ be an almost-directed $d$-homological pair such that $\mathrm{gl.dim}(A)\leq d$. Then there is a bijection between $d$-strong torsion classes in $\mathcal{C}$ and maximal support pre-$d$-tilting $A$-modules.
\end{conj}

If $T$ is a proper support-$d$-tilting module, then we call the $d$-strong torsion class $\mathcal{T}=\mathrm{Fac}(T)\cap \mathcal{C}$ a \emph{standard $d$-strong torsion class}. Likewise, we call the $d$-strong torsion-free class $\mathcal{F}=\mathrm{Sub}(T)\cap\mathcal{C}$ a \emph{standard $d$-strong torsion-free class}.

\section{Wide subcategories} 

Let $(A,\mathcal{C})$ be a $d$-homological pair. As introduced in \cite[Definition 2.11]{hjv}, an additive subcategory $\mathcal{W}\subseteq \mathcal{C}$ is a \emph{wide subcategory} if it satisfies the following conditions:

\begin{enumerate}[(W1)]
\item For each $W_1, W_2\in \mathcal{W}$ and morphism $W_1\rightarrow W_2$ there exist $d$-exact sequences
 $$0\rightarrow M_1\rightarrow M_2\rightarrow \cdots \rightarrow M_d\rightarrow W_1\rightarrow W_2$$ 
 $$W_1\rightarrow W_2\rightarrow N_1\rightarrow \cdots \rightarrow N_{d-1}\rightarrow N_d\rightarrow 0$$ 
such that $M_i,N_i\in\mathcal{W}$ for all $1\leq i \leq d$.\label{w1}
\item For any $W, W^\prime \in \mathcal{W}$, then any $d$-exact sequence
$$0\rightarrow W\rightarrow U_1\rightarrow U_2 \rightarrow \cdots \rightarrow U_d\rightarrow W^\prime \rightarrow 0$$ is Yoneda equivalent to a $d$-exact sequence 
$$0\rightarrow W\rightarrow W_1\rightarrow W_2 \rightarrow \cdots \rightarrow W_d\rightarrow W^\prime \rightarrow 0$$
such that $W_i\in \mathcal{W}$ for all $1\leq i\leq d$.\label{w2}
\end{enumerate}

We need to supplement this with the following definition.

\begin{defn}\label{resonant}
Let $(A,\mathcal{C})$ be a $d$-homological pair. For two wide subcategories $\mathcal{W}_1\subseteq \mathcal{C}$ and $\mathcal{W}_2\subseteq \mathcal{C}$ such that $\mathcal{W}_1\ne \mathcal{W}_2$, we say $\mathcal{W}_1\prec \mathcal{W}_2$ if for all $M\in\mathcal{W}_1$ and $N\in\mathcal{W}_2$ such that $M\not\cong N$, then  $\mathrm{Hom}_A(N,M)=0$. 

A collection $(\mathcal{W}_1,\mathcal{W}_2,\ldots ,\mathcal{W}_n)$ of wide subcategories of $\mathcal{C}$ is a \emph{directed collection of wide subcategories in $\mathcal{C}$} if $\mathcal{W}_i\prec \mathcal{W}_{j}$ for all $1\leq i<j\leq n$.

A directed collection of wide subcategories $(\mathcal{W}_1,\mathcal{W}_2,\ldots ,\mathcal{W}_n)$ is \emph{resonant} if the union $\mathcal{W}_1 \cup \mathcal{W}_2\cup \cdots\cup \mathcal{W}_n$ has an additive generator $M$ such that $\mathrm{Fac}(M)\cap\mathcal{C}$ is a standard $d$-strong torsion class.

Dually, a directed collection of wide subcategories $(\mathcal{W}_1,\mathcal{W}_2,\ldots ,\mathcal{W}_n)$ is \emph{coresonant} if the union $\mathcal{W}_1 \cup \mathcal{W}_2\cup \cdots\cup \mathcal{W}_n$ has an additive generator $M$ such that $\mathrm{Sub}(M)\cap \mathcal{C}$ is a standard $d$-strong torsion-free class.
\end{defn}

Now fix a positive integer $n$, and let $(A,\mathcal{C})$ be a $d$-homological pair such that $A$ is a $d$-Auslander algebra of linearly oriented type $A_n$. Observe that for any $d$-exact sequence in $\mathcal{C}$, $$0\rightarrow M_0 \rightarrow M_1\rightarrow \cdots \rightarrow M_{d+1}\rightarrow 0$$ and module $N\in\mathcal{C}$, there is a $d$-exact sequence for every $0\leq i\leq d$
$$0\rightarrow M_0 \rightarrow M_1\rightarrow \cdots \rightarrow M_i\oplus N\rightarrow M_{i+1}\oplus N\rightarrow \cdots \rightarrow M_{d+1}\rightarrow 0.$$ Unless otherwise stated, we will consider every $d$-exact sequence to have no such summand $N$.
The following result is implicit in the proof of \cite[Theorem 2.18]{jk} using iterated application of Lemma 1.20 \cite{jk}.

\begin{lm}\label{indec}
Let $A$ be a $d$-Auslander algebra of linearly oriented type $A_n$ with unique $d$-cluster-tilting subcategory $\mathcal{C}$. Then for any $d$-exact sequence in $\mathcal{C}$ consisting of only indecomposable modules:
$$0\rightarrow X_0\rightarrow X_1\rightarrow \cdots\rightarrow X_{d+1}\rightarrow 0$$ there exists a properly-supporting idempotent $e$, such that for all $0\leq i\leq d$ each $X_i$ is injective 
as an $A/\langle e \rangle$-module, and for all $1\leq i\leq d+1$ each $X_i$ is projective 
as an $A/\langle e \rangle$-module.
\end{lm}

\begin{defn}\label{betadef}
Let $A$ be a $d$-Auslander algebra of linearly oriented type $A_n$ with unique $d$-cluster-tilting subcategory $\mathcal{C}$, and let $\mathcal{T}\subseteq \mathcal{C}$ be $d$-strong torsion class. Then define $M\in \alpha(\mathcal{T})$ if for any submodule $M^\prime$ of $M$ such that $M^\prime \in \mathcal{C}$ and indecomposable $K_0\in \mathcal {T}$ with surjective morphism $f:K_0\rightarrow M^\prime$, then there exists a kernel of $f$ in $\mathcal{C}$: $$0\rightarrow K_d\rightarrow\cdots \rightarrow K_2\rightarrow K_1\rightarrow K_0\rightarrow M^\prime\rightarrow 0$$ such that $K_i\in \mathcal{T}$ for all $0\leq i\leq d$.
\end{defn}

\begin{lm}\label{indec2}
Let $A$ be a $d$-Auslander algebra of linearly oriented type $A_n$ with unique $d$-cluster-tilting subcategory $\mathcal{C}$. Let $T$ be a $d$-tilting $A$-module and $\mathcal{T}$ the associated standard $d$-strong torsion class. Suppose that
$$0\rightarrow M_0\rightarrow M_1\rightarrow \cdots \rightarrow M_{d+1}\rightarrow 0$$ is an exact sequence such that $M_i\in\mathcal{C}$ is an indecomposable module for each $0\leq i\leq d+1$. If $M_0, M_{d+1}\in\alpha(\mathcal{T})$, then $M_i\in\alpha(\mathcal{T})$ for all $0\leq i\leq d+1$.
\end{lm}

\begin{proof}
Observe that there is always some idempotent $e$ such that every projective $A/\langle e \rangle$-module is also projective as an $A$-module, as well as an idempotent $f$ such that every injective $A/\langle f \rangle$-module is also injective as an $A$-module. In the second case, for any standard $d$-strong torsion class $\mathcal{T}$ in $\mathcal{C}$, then $\mathcal{T}\cap \mathrm{mod}(A/\langle f \rangle)$ is a standard $d$-strong torsion class in $\mathcal{C}\cap \mathrm{mod}(A/\langle f \rangle)$.

By Lemma \ref{indec}, there exists some properly-supporting idempotent $e$ where \\$M_1,\ldots, M_{d}$ are projective-injective $A/\langle e \rangle$-modules.  By the above argument we may assume that each $M_i$ is projective as an $A$-module for $0\leq i\leq d$. As $M_0, M_{d+1}\in\mathcal{T}$ and $\mathrm{Ext}^d_A(M_{d+1},M_0)\ne0$ there must be some $X\in\mathrm{add}(T)$ such that $M_{d+1}\not\cong X$ and there is a surjection $X\twoheadrightarrow M_{d+1}$. There is also a surjection $M_{d}\twoheadrightarrow M_{d+1}$. So there must either be a surjection $M_d\twoheadrightarrow X$ or $X\twoheadrightarrow M_d$. Now if there is a surjection $M_{d}\twoheadrightarrow X$, then we may form a resonance diagram and see that $\mathrm{Ext}^d_A(X,M_0)\ne 0$, a contradiction. Therefore there must be a surjection $X\twoheadrightarrow M_{d}$ and hence $M_d\in\mathcal{T}$. 
 Lemma \ref{chevelle} then implies that $M_i\in\mathcal{T}$ for each $0\leq i\leq d+1$.
 
  Now suppose there is a surjection from $T^\prime_i\twoheadrightarrow M_i$ for some $0<i<d+1$ and $T^\prime_i\in\mathcal{T}$. Then we may form a resonance diagram such that there are surjections $T^\prime_j\twoheadrightarrow M_j$ where $T^\prime_j\in\mathcal{T}$ for all $0\leq j\leq d+1$ . Since both $M_0,M_{d+1}\in\alpha(\mathcal{T})$, applying the definition as well as Lemma \ref{chevelle}, we must have that $M_i\in\alpha(\mathcal{T})$ for all $0\leq i\leq d+1$. The proof is similar for a surjection from $T^\prime_i$ onto a submodule of $M_i$. 
\end{proof}

We are now ready to show our main result, a generalisation of Theorem \ref{crowd}.

\begin{theorem}\label{masod}
Let $A$ be a $d$-Auslander algebra of linearly oriented type $A_n$ with unique $d$-cluster-tilting subcategory $\mathcal{C}$. Then there are bijections between the following:
\begin{enumerate}
\item Proper support-$d$-tilting modules. \label{masod1}
\item Standard $d$-strong torsion classes in $\mathcal{C}$. \label{masod2}
\item Resonant collections of wide subcategories in $\mathcal{C}$.\label{masod3}
\item Standard $d$-strong torsion-free classes in $\mathcal{C}$.\label{masod4}
\item Coresonant collections of wide subcategories in $\mathcal{C}$.\label{masod5}
\end{enumerate}
\end{theorem}

\begin{proof}

(\ref{masod1})$\iff$ (\ref{masod2}) By Theorem \ref{elso}, for any proper support-$d$-tilting module $T$, the class $\mathcal{T}:=\mathrm{Fac}(T)\cap \mathcal{C}$ is a standard $d$-strong torsion class in $\mathcal{C}$. Conversely, given a standard $d$-strong torsion class $\mathcal{T}$, consider the $\mathrm{Ext}$-projectives in $\mathcal{T}$; that is the class of modules 
$$\{M\in\mathcal{T}|\mathrm{Ext}_A^d(M,\mathcal{T})=0\}.$$
It can be seen that $T$ is an additive generator of this class, and hence each standard $d$-strong torsion class determines a unique  proper support-$d$-tilting module.
(\ref{masod2})$\iff$ (\ref{masod3})

For a proper support-$d$-tilting module, we wish to find a directed collection of wide subcategories that comprise $\alpha(\mathcal{T})$. For two indecomposable modules\\ $M,N\in\alpha(\mathcal{T})$, define $M\prec N$ if there exists a $d$-exact sequence in $\mathcal{C}$
$$0\rightarrow X_0\rightarrow X_1\rightarrow \cdots \rightarrow M\rightarrow N\rightarrow \cdots\rightarrow X_{d+1}\rightarrow 0$$
where at least one of $X_0$ and $X_{d+1}$ is not in $\alpha(\mathcal{T})$. This definition may be extended transitively.

Now divide $\alpha(\mathcal{T})$ into classes such that any module $M_1,M_2\in\alpha(\mathcal{T})$ are in different classes whenever $M_1\prec M_2$ and extend additively.
Note that we allow there to exist modules $M_1,M_2,N\in\alpha(\mathcal{T})$ such that $M_1\not \prec N$ $N\not \prec M_1$, $M_2\not\prec N$, $N\not \prec M_1$ and $M_1\prec M_2$. In this case, we allow $N$ to be in the same class as $M_1$ as well as the same class as $M_2$.

By construction, each of these classes is a wide subcategory of $\mathcal{C}$: by definition condition (W\ref{w1}) holds. In addition, Lemma \ref{indec2} implies that condition (W\ref{w2}) holds: any sequence $$\phi:0\rightarrow M_0\rightarrow M_1\rightarrow \cdots \rightarrow  M_{d+1}\rightarrow 0$$ such that $M_0, M_{d+1}\in \alpha(\mathcal{T})$ and $M_i\not\in \alpha(\mathcal{T})$ must satsify $M_{d+1}\prec M_0$ . Assume that this is not true: there must be a summand $N$ of $M_d$ such that there is a surjection $N\twoheadrightarrow M_{d+1}$. There is some $T_0 \in\mathrm{add}(T)$ and there are surjections $$N\twoheadrightarrow T_0\twoheadrightarrow M_{d+1},$$ where we allow $N\cong T_0$ (but not $M_{d+1}\cong T_0$ as this would contradict the $\mathrm{Ext}$-projectivity of $T$ in $\mathcal{T}$). Since $M_{d+1}\in\alpha(T)$, there is a $d$-exact sequence 
$$0\rightarrow T_{d}\rightarrow \cdots \rightarrow T_1\rightarrow T_0\rightarrow M_{d+1}\rightarrow 0$$ such that $T_i\in \alpha(\mathcal{T})$ for all $0\leq i\leq d$.
This induces an exact sequence $$\mathrm{Hom}_A(T_d,M_0)\rightarrow \mathrm{Ext}^d_A(M_{d+1},M_0)\rightarrow \mathrm{Ext}^d_A(T_0,M_0)=0.$$ Therefore $\mathrm{Hom}_A(T_d,M_0)\ne 0$ and this morphism together with $\phi$ induces a resonance diagram wherein $M_{d+1}\prec M_{0}$. 
We claim in addition that $\prec$ is a total order: suppose there is a $d$-exact sequence $$\phi_N:0\rightarrow N_0\rightarrow N_1\rightarrow \cdots\rightarrow N_{d+1}\rightarrow 0$$ such that $N_i\in\alpha(\mathcal{T})$ for all $0\leq i\leq d+1$ apart from some $0\leq l\leq d+1$. Suppose further that there is a module $M_j\in\alpha(\mathcal{T})$  such that $M_{j}\prec N_i$ for all $1\leq i\leq d+1$: there is a $d$-exact sequence  $$\phi_M:0\rightarrow M_0\rightarrow M_1\rightarrow \cdots\rightarrow M_{d+1}\rightarrow 0$$ such that $M_{k}\cong N_l$ for some $k>j$. Now form a resonance diagram - by Lemma \ref{indec} there exists a properly-supporting idempotent $e$ such that $M_0,M_1,\ldots, M_d$ and $N_0,N_1,\ldots, N_d$ are projective as $A/\langle e \rangle$-modules and $M_1,M_2,\ldots, M_{d+1}$ as well as $N_1,N_2,\ldots, N_{d+1}$ are projective as $A/\langle e \rangle$-modules. Also Lemma \ref{skel2} implies that there is some functor $F$ such that $F(T)$ is a tilting $A/\langle e \rangle$-module. Assume that $M_j$ is projective-injective as an $A/\langle e\rangle$-module and that there is a sequence in the chosen resonance diagram 
$$\phi_X:0\rightarrow X_0\rightarrow \cdots \rightarrow M_j\rightarrow \cdots \rightarrow N_i\rightarrow \cdots\rightarrow X_{d+1}\rightarrow 0$$
for some $0< l<i$. We have that $M_0$ is a projective $A/\langle e\rangle$-module, and so there must be a $F(T)$-coresolution of $X_0$.  If $X_ 0\not\in \mathrm{add}(F(T))$, then the $F(T)$-coresolution must factor through $M_j$ - a contradiction. We still have to check in case modules that are in $\mathrm{add}(F(T))$ might not be in $\alpha(\mathcal{T})$. So assume that $M_0\in\alpha(\mathcal{T})$ and that every projective $A/\langle e\rangle$-module is projective as an $A$-module. The module $M_1$ is projective as an $A$-module. Furthermore there is a morphism $M_1\rightarrow N_1$. If $M_1\not\in\mathrm{add}(T)$ then there is some $T^\prime \in\mathrm{add}(T)$ such that $\mathrm{Ext}^d_A(T^\prime, M_1)\ne0$. Within the resonance diagram containing $\phi_M$ and $\phi_N$, there is a further $d$-exact sequence $$\phi_L:0\rightarrow L_0\rightarrow L_1 \rightarrow\cdots\rightarrow  L_{d+1}\rightarrow 0$$and exact sequences $L_i\rightarrow M_i\rightarrow N_i$ for $0\leq i\leq k-2$. We must have isomorphisms $\mathrm{Ext}^d_A(T^\prime,L_1)\cong \mathrm{Ext}^d_A(T^\prime, M_1)\cong \mathrm{Ext}^d_A(T^\prime, M_2)$. Since $M_0,N_1\in\mathrm{add}(T)$, this means the $d$-exact sequence from $L_1$ to $T^\prime$ factors through $N_1$. Then either the terms in a cokernel of the induced morphism $X\rightarrow N_2$ (contained in the resonance diagram) are included in the $\mathrm{add}(T)$ coresolution of $N_2$, or there is another $T^{\prime\prime}\in\mathrm{add}(T)$ such that $\mathrm{Ext}^d_A(T^{\prime\prime}, M_2)\ne 0$. Apply this argument along $\phi_M$ - at some point $M_k\cong N_l\in\alpha(\mathcal{T})$. Therefore $\mathrm{Ext}^d_A(T, M_0)\ne0$, as the terms in the cokernel of $L_{k-2}\rightarrow M_{k-1}$ must be contained in the $\mathrm{add}(T)$-coresolution of $M_{k-1}$. This is a contradiction.

Conversely assume that every injective $A/\langle e\rangle$-module is injective as an $A$-module. Form a resonance diagram between $\phi_N$ and the projective covers of $N_i$ \\($0\leq i\leq d+1$), which are also injective modules for all $1\leq i\leq d+1$ and hence in $\mathcal{T}$. This induces a $d$-exact sequence in $\alpha(\mathcal{T})$:
$$\phi_{N}^\prime:0\rightarrow  \Omega^d(N_1)\rightarrow \cdots \rightarrow\Omega^d(N_{d+1})\rightarrow N_0\rightarrow 0.$$
Form a resonance diagram with $\phi_N^\prime $ and the morphism $M_0\rightarrow N_0$ and apply the previous argument. This implies that $N_i\in\alpha(\mathcal{T})$ for all $0\leq i\leq d+1$.

Now let $T^\prime$ be the minimal summand of $T$ such that $\mathrm{Fac}(T^\prime)=\mathrm{Fac}(T)$. By definition, $T^\prime\in\alpha(\mathcal{T})$. Now
let $(\mathcal{W}_1,\ldots,\mathcal{W}_n)$ be the directed collection of wide subcategories defined above. Let $M$ be an additive generator of $\mathcal{W}_1\cup\cdots\cup \mathcal{W}_n$ - $M$ is in fact an additive generator of $\alpha(\mathcal{T})$. Therefore $\mathrm{Fac}(M)=\mathrm{Fac}(T^\prime)$, and $(\mathcal{W}_1,\ldots,\mathcal{W}_n)$ is a resonant collection of wide subcategories.

The bijections (\ref{masod1})$\iff$ (\ref{masod4})  and  (\ref{masod1})$\iff$ (\ref{masod5}) are dual. 
\end{proof}

\section{Examples and further directions}
	
\begin{eg} As in Example \ref{eg1}, let $A$ be the algebra 
$$
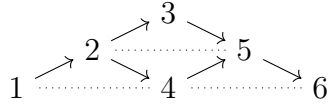
\begin{tikzpicture}[xscale=5,yscale=2.5]
	\node(xa) at (-2,1){$1$};
	\node(xb) at (-1.6,1){$4$};
	\node(xc) at (-1.2,1){$6$};
	
	\node(xab) at (-1.8,1.2){$2$};
	\node(xac) at (-1.4,1.2){$5$};
	
	\node(xbb) at (-1.6,1.4){$3$};
	
	\draw[->](xa) edge(xab);
	\draw[->](xb) edge(xac);
	
	\draw[->](xbb) edge(xac);
		
	\draw[->](xab) edge(xb);
	\draw[->](xac) edge(xc);
	
	\draw[->](xab) edge(xbb);
	
	\draw[-,dotted](xc) edge(xb);
	\draw[-,dotted](xb) edge(xa);
	\draw[-,dotted](xac) edge(xab);

	\end{tikzpicture}$$
	The support-$2$-tilting modules form a lattice under inclusion of standard $d$-strong torsion classes, as depicted in Figure \ref{stdiag}. The module $T=P_1\oplus P_2\oplus P_3\oplus P_4\oplus P_5\oplus S_4$ is a proper support-$2$-tilting $A$-module, and its associated $2$-strong torsion class consists of $$\bigg \langle\begin{matrix} 1\\2\\3\end{matrix}\oplus  \begin{matrix} 1\\2\end{matrix}\oplus  \begin{matrix} 1\end{matrix}\oplus  \begin{matrix} &2&\\3&&4\\&5&\end{matrix}\oplus  \begin{matrix} 2\\4\end{matrix}\oplus  \begin{matrix} 3\\5\\6\end{matrix}\oplus  \begin{matrix} 4\\5\end{matrix}\oplus  \begin{matrix} 4\end{matrix}\oplus  \begin{matrix} 5\\6\end{matrix} \bigg \rangle.$$

The resonant collection of wide subcategories of $\mathcal{C}$ consists of $$\mathcal{W}_1=\bigg \langle \begin{matrix} 1\\2\end{matrix}\oplus  \begin{matrix} 1\\2\\3 \end{matrix}\oplus  \begin{matrix}3\\5\\6\end{matrix}\oplus\begin{matrix} 5\\6\end{matrix} \bigg \rangle.$$
$$\mathcal{W}_2=\bigg \langle \begin{matrix} 1\end{matrix}\oplus  \begin{matrix} 1\\2\\3 \end{matrix}\oplus  \begin{matrix} &2&\\3&&4\\&5&\end{matrix}\oplus\begin{matrix} 4\\5\end{matrix} \bigg \rangle.$$

In fact, this collection of wide subcategories is also coresonant, under the support-$2$-cotilting module $$C=\begin{matrix} 1\\2\\3 \end{matrix} \oplus \begin{matrix} 1\\2\end{matrix}\oplus  \begin{matrix} 1 \end{matrix}\oplus  \begin{matrix} &2&\\3&&4\\&5&\end{matrix}\oplus  \begin{matrix}3\\5\\6\end{matrix}\oplus\begin{matrix} 6\end{matrix}$$
\end{eg}

\begin{eg}\label{eg2}
For example, take $A$ to be the algebra with quiver
	$$\begin{tikzpicture}[xscale=2,yscale=2.5]
	\node(a) at (-2,1){$1$};
	\node(b) at (-1.5,1){$2$};
	\node(c) at (-1,1){$3$};
    \node(d) at (-0.5,1){$4$};
	\node(e) at (0,1){$5$};
	\node(f) at (0.5,1){$6$};
    \node(g) at (1,1){$7$};
	\draw[->](a) edge(b);
	\draw[->](b) edge(c);
	\draw[->](c) edge(d);
	\draw[->](d) edge(e);
	\draw[->](e) edge(f);
	\draw[->](f) edge(g);
	\end{tikzpicture}$$
and with relations given by all paths of length 3. 
This is an example from a class of $d$-representation-finite algebras was introduced in \cite[Section 5]{vaso}. Furthermore, the wide subcategories of these algebras have been classified in \cite{hjv}, see also \cite{fed} for further study of wide subcategories of such algebras.
Then  $$\mathcal{C}=\bigg \langle P_1\oplus P_2\oplus P_3\oplus P_4\oplus P_5\oplus P_6\oplus P_7\oplus I_1\oplus I_2 \bigg \rangle$$ is a $2$-cluster-tilting subcategory of $\mathrm{mod}(A)$. 
There are five proper support-$2$-tilting $A$-modules: $S_1$, $S_7$,  $A$ and the following two modules $$T_1:=P_1\oplus P_2\oplus P_3\oplus P_4\oplus P_5\oplus P_6 \oplus I_2;$$ $$T_2:=P_1\oplus P_2\oplus P_3\oplus P_4\oplus P_5\oplus I_1\oplus I_2.$$ A similar process to Theorem \ref{masod} shows that $T_1$ determines a resonating collection of wide subcategories: $(P_6,P_5,P_4,P_3,P_2,P_1)$ and $T_2$ determines the resonating collection of wide subcategories $(P_5,P_4,P_3,P_2,P_1)$. 
\end{eg}

Observe that for the algebra $A$ in Example \ref{eg2}, any wide subcategory that is contained in a resonating collection of wide subcategories generated by either the whole of $\mathcal{C}$, or by precisely one module. The classification of wide subcategories for this algebra that given in \cite{hjv} tells us that not all wide subcategories are of such a form. Restricting our focus, we ask the following question:

\begin{qu}
Let $(A,\mathcal{C})$ be an almost-directed $d$-homological pair. Can all wide subcategories be obtained from a resonant collection of wide subcategories associated with a proper support-$d$-tilting module?
\end{qu}

Another question we may ask is the following:

\begin{qu}
Let $(A,\mathcal{C})$ be an almost-directed $d$-homological pair such that\\ $\mathrm{gl.dim}(A)\leq d$. Is it possible to classify the directed collections of wide subcategories that are both resonant and coresonant?
\end{qu}

\section{Acknowledgements}
This paper was completed as part of the author's PhD studies, with the support of the Austrian Science Fund (FWF): W1230. I would like to thank my supervisor, Karin Baur, for her continued help and support during my studies. This project originated during a research visit to Newcastle University, and further work was completed on a second research visit to University of Paris-Sud. The author would like to thank both universities for their generous hospitality, and in particular Peter J\o rgensen for introducing him to the article \cite{hjv}, as well as many fruitful conversations; and also Pierre-Guy Plamondon and Sondre Kvamme for interesting discussions.

	\begin{figure} 
		\caption{The support-$2$-tilting lattice for the $2$-Auslander algebra of linearly-oriented type $A_3$.}
	$$\begin{tikzpicture}[scale=2]\label{stdiag}
    \node(a) at (0,0.5){$\emptyset$};

	\node(b) at (2.5,1){$\begin{smallmatrix} 1\end{smallmatrix}$};
   
    \draw[-](a) --(b) ;
	
	\node(c) at (0,1.7){$\begin{smallmatrix} 4\end{smallmatrix}$};

	\draw[-](a) -- (c) ;
	
	\node(d) at (-3,1.5){$ \begin{smallmatrix}6\end{smallmatrix}$};

	\draw[-](a) --(d) ;
	
	\node(e) at (-0.5,3.1){$ \begin{smallmatrix} 1\end{smallmatrix}\  \begin{smallmatrix} 6\end{smallmatrix}$};
	
	\draw[-] (b)--(e) ;
	
	\draw[-] (d)--(e) ;
	
	\node(f) at (-2.5,4.5) {$ \begin{smallmatrix} 4\\ 5\end{smallmatrix}\ \begin{smallmatrix} 4\end{smallmatrix}\  \begin{smallmatrix} 5\\6\end{smallmatrix}$};
	
	\draw[-] (c) to [bend left=10] (f);

\node(g) at (-3,7.5){$  \begin{smallmatrix} 4\\ 5\end{smallmatrix}\  \begin{smallmatrix} 5\\6\end{smallmatrix}\ \begin{smallmatrix} 6\end{smallmatrix}$};

\draw[-] (f) -- (g) ;

\draw[-] (d) to [bend left= 15]  (g) ;

\node(h) at (2.5,2.2) {$\begin{smallmatrix} 1\\ 2\end{smallmatrix}\ \begin{smallmatrix} 1\end{smallmatrix}\  \begin{smallmatrix} 2\\ 4\end{smallmatrix}$};

\draw (b) -- (h) ;

\node(i) at (1.2,3.5) {$\begin{smallmatrix} 1\\ 2\end{smallmatrix}\ \begin{smallmatrix} 2\\ 4\end{smallmatrix}\ \begin{smallmatrix} 4\end{smallmatrix}\ $};

\draw (c) to  (i);

\draw (h) to  (i);

\node(j) at (2,7.3){$ \begin{smallmatrix} 1\\2\\3\end{smallmatrix}\  \begin{smallmatrix} 1\end{smallmatrix}\ \begin{smallmatrix} &2&\\3&&4\\&5&\end{smallmatrix}\  \begin{smallmatrix} 3\\5\\6\end{smallmatrix}\ \begin{smallmatrix} 5\\6\end{smallmatrix}\ \begin{smallmatrix} 6 \end{smallmatrix}\ $};

\node(q) at (2, 6){$\begin{smallmatrix} 1\\2\\3\end{smallmatrix}\ \begin{smallmatrix} 1\\2\end{smallmatrix}\ \begin{smallmatrix} 1\end{smallmatrix}\ \begin{smallmatrix} &2&\\3&&4\\&5&\end{smallmatrix}\  \begin{smallmatrix}3\\ 5\\6\end{smallmatrix}\   \begin{smallmatrix} 6 \end{smallmatrix}$};

\draw (q) to (j); 

\draw (e) to [bend left=30]  (q);

\node(k) at (-1.8,6.8) {$ \begin{smallmatrix} 1\\2\\3\end{smallmatrix}\  \begin{smallmatrix} &2&\\3&&4\\&5&\end{smallmatrix}\  \begin{smallmatrix} 3\\5\\6\end{smallmatrix}\  \begin{smallmatrix} 4\\5\end{smallmatrix}\  \begin{smallmatrix} 4\end{smallmatrix}\ \begin{smallmatrix} 5\\6\end{smallmatrix}$};

\draw (f) to (k); 

\node(m) at (0,8){$ \begin{smallmatrix} 1\\2\\3\end{smallmatrix}\  \begin{smallmatrix} &2&\\3&&4\\&5&\end{smallmatrix}\  \begin{smallmatrix} 3\\5\\6\end{smallmatrix}\  \begin{smallmatrix} 4\\5\end{smallmatrix}\ \begin{smallmatrix} 5\\6\end{smallmatrix}\ \begin{smallmatrix} 6\end{smallmatrix}$};

\draw (g) to[bend left=10]  (m); 

\draw (k) to  (m); 

\draw (j) to [bend right=10] (m);

\node(n) at (2.5,4){$ \begin{smallmatrix} 1\\2\\3\end{smallmatrix}\ \begin{smallmatrix} 1\\2\end{smallmatrix}\ \begin{smallmatrix} 1\end{smallmatrix}\ \begin{smallmatrix} &2&\\3&&4\\&5&\end{smallmatrix}\ \begin{smallmatrix} 2\\4\end{smallmatrix}\  \begin{smallmatrix} 3\\5\\6\end{smallmatrix}$};

\node(o) at (1.25,4.95){$ \begin{smallmatrix} 1\\2\\3\end{smallmatrix}\ \begin{smallmatrix} 1\\2\end{smallmatrix}\ \begin{smallmatrix} &2&\\3&&4\\&5&\end{smallmatrix}\ \begin{smallmatrix} 2\\4\end{smallmatrix}\  \begin{smallmatrix} 3\\5\\6\end{smallmatrix}\  \begin{smallmatrix} 4\end{smallmatrix}$};

\node(p) at (-0.8,5.8){$ \begin{smallmatrix} 1\\2\\3\end{smallmatrix}\ \begin{smallmatrix} &2&\\3&&4\\&5&\end{smallmatrix}\ \begin{smallmatrix} 2\\4\end{smallmatrix}\  \begin{smallmatrix} 3\\5\\6\end{smallmatrix}\  \begin{smallmatrix} 4\\5\end{smallmatrix}\ \begin{smallmatrix} 4\end{smallmatrix}$};

\draw (o) to  (p); 

\draw (p) to  (k);

\draw (n) to [bend right =10]  (q);

	\draw[-] (h) to    (n); 

	\draw[-] (i) to    (o);

	\draw[-] (n) to    (o);

\end{tikzpicture}$$

\end{figure}
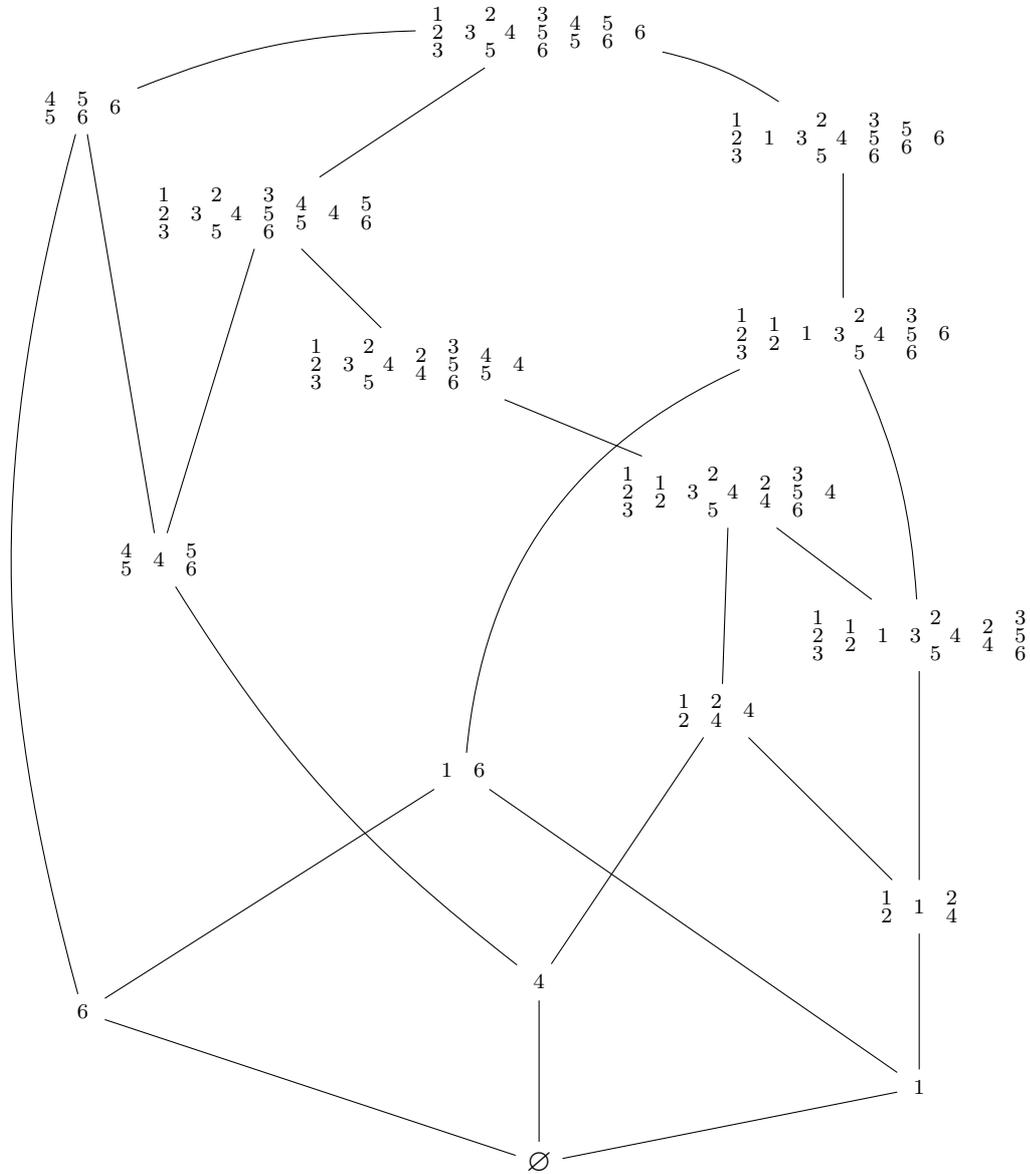

	

\bibliographystyle{amsplain}
\bibliography{higher}

\end{document}